\documentclass{amsart}
\usepackage{amsmath}
\usepackage{amsthm}
\usepackage[applemac]{inputenc} 
\usepackage{hyperref}
\usepackage{nicematrix}
\usepackage{amssymb}
\usepackage{graphicx}
\usepackage{enumitem}
\usepackage{amsfonts}
\usepackage[all]{xy}
\usepackage{array}
\usepackage{graphicx}
\usepackage{xcolor}
\usepackage{ytableau}

\usepackage[textsize=tiny]{todonotes}

\usepackage{enumitem}

\newcommand\lra{\longrightarrow}

\newcommand{\ol}{\overline}

\newcommand{\C}{\mathbb{C}}

\newcommand{\G}{\mathbb{G}}

\renewcommand{\P}{\mathbb{P}}
\newcommand{\PP}{\mathbb{P}}

\newcommand{\R}{\mathbb{R}}

\newcommand{\Z}{\mathbb{Z}}

\newcommand{\cA}{\mathcal A}

\newcommand{\cO}{\mathcal O}
\newcommand{\cP}{\mathcal P}

\newcommand{\cU}{\mathcal U}

\newcommand{\cY}{\mathcal Y}

\newcommand{\DA}{{\rm A}}
\newcommand{\DB}{{\rm B}}

\newcommand{\DD}{{\rm D}}

\newcommand{\SL}{\operatorname{SL}}
\newcommand{\SO}{\operatorname{SO}}

\let\div\relax
\DeclareMathOperator{\div}{div}

\newcommand{\CU}{\mathcal{C}\mkern-1mu U}

\newcommand{\CX}{\mathcal{C}\mkern-1mu X}

\newcommand{\GU}{\mathcal{G} U}
\newcommand{\GX}{\mathcal{G}\mkern-1mu X}

\newcommand{\LX}{\mathcal{L}\!X}

\renewcommand{\ss}{\mathrm{ss}}

\DeclareMathOperator{\Bir}{Bir}

\DeclareMathOperator{\Chow}{Chow}

\DeclareMathOperator{\Conv}{Conv}

\DeclareMathOperator{\HH}{H}

\DeclareMathOperator{\Proj}{Proj}

\DeclareMathOperator{\Spec}{Spec}

\DeclareMathOperator{\Mo}{M}

\DeclareMathOperator{\No}{N}

\newcolumntype{C}{>{$}c<{$}}

\newcolumntype{C}{>{$}c<{$}}

\newtheorem{theorem}{Theorem}[section]

\newtheorem{lemma}[theorem]{Lemma}
\newtheorem{proposition}[theorem]{Proposition}
\newtheorem{corollary}[theorem]{Corollary}

\theoremstyle{definition}
\newtheorem{definition}[theorem]{Definition}

\newtheorem{remark}[theorem]{Remark}
\newtheorem{construction}[theorem]{Construction}

\title{A two-step approach to Chow quotients}
\author[Occhetta]{Gianluca Occhetta}
\address{Dipartimento di Matematica, Universit\`a degli Studi di Trento, via
Sommarive 14 I-38123, Trento (TN), Italy}
\email{gianluca.occhetta@unitn.it, eduardo.solaconde@unitn.it}

\author[Sol\'a Conde]{Luis E. Sol\'a Conde}
\thanks{Authors partially supported by INdAM--GNSAGA}

\begin{document}

\begin{abstract} 
The Chow quotient of a projective variety by the action of a complex torus is known to have a very complicated geometry, even in the case of simple varieties, such as rational homogeneous varieties. In this paper we propose an approach in which the geometry of the Chow quotient is encoded in a projective toric variety and a finite subgroup  of its birational automorphisms. We then illustrate how to apply our strategy in the case of some particular rational homogeneous varieties.
\end{abstract}
\maketitle

\section{Introduction}\label{sec:intro}

The construction and study of quotients by algebraic group actions is a central theme in algebraic geometry, with deep connections to birational geometry and the theory of moduli spaces. Among the various notions of quotient, Chow quotients, introduced by Kapranov~\cite{Kapranov}, occupy a distinguished position: while they are often difficult to compute explicitly, they tend to encode rich geometric information and admit meaningful modular interpretations in special cases. Classical examples include their relationship with Grothendieck--Knudsen moduli spaces~\cite{Kn}, as well as with spaces of complete collineations and complete quadrics (see, e.g.,~\cite{DeConciniProcesi}).

In this paper, we investigate the problem of constructing Chow quotients of smooth projective varieties under the action of algebraic tori. Our approach is motivated by the observation that, despite their global nature, Chow quotients often admit a description in terms of simpler, more combinatorial objects arising from torus-invariant open subsets.
 
This perspective was explored in our previous work~\cite{BOS2}, where we studied the action of a maximal torus on the complete flag variety of $\P^3$. In that setting, we showed that, in favorable situations, the computation of the Chow quotient can be naturally divided into two steps. First, one computes certain combinatorial quotients associated with torus-invariant open subsets. Second, one analyzes the birational relations among these quotients in order to reconstruct the global geometry. In the case of rational homogeneous varieties endowed with the action of a maximal torus in their automorphism group, these birational relations are governed by the action of the corresponding Weyl group.

The main goal of the present paper is to identify general conditions under which this strategy applies. More precisely, we establish a set of hypotheses ensuring that the Chow quotient can be reconstructed as an inverse limit of combinatorial quotients associated with torus-invariant open subsets. This provides a conceptual and effective framework for computing Chow quotients in a relevant class of examples, namely those in which fixed points are isolated and their weights with respect to a given linearization on an ample line bundle are extremal; following \cite{BB}, we say that the action is fully definite on every fixed point:

\begin{theorem}\label{thm:main}
Let $X$  be a  smooth projective variety, endowed with the action of an algebraic torus $H$, fully definite on every fixed point. Then, for every fixed point $x\in X$ there exists an $H$-equivariant isomorphism of $T_{X,x}$ with an open affine neighborhood $U_x$ of $x$ in $X$, whose combinatorial quotients $\CU_x$ are projective. Moreover there exists a natural morphism from the Chow quotient of $X$ to $\prod_{x\in X^H}\CU_x$, which is an isomorphism to the normalization of its image.
\end{theorem}

The hypotheses of the Theorem are satisfied in the case of rational homogeneous varieties endowed with the action of a maximal torus, but also in the case of some spherical varieties (with the action of appropriate tori). As an application of this result, we consider the case of $\P(T_{\P^n})$, quadrics and Grassmannians. . 
In particular, we recover the Chow quotients of Grassmannians of lines --originally computed by Kapranov~\cite{Kapranov}-- from our perspective, obtaining a streamlined and conceptually transparent derivation. 

\noindent{\bf Outline: }
The paper is organized as follows. In Section \ref{sec:prelim} we recall the necessary background on Chow quotients and torus actions. Section \ref{sec:main} contains the proof of our main statement relating Chow and combinatorial quotients of torus invariant neighborhoods of points in which the action is fully definite. Finally, in Sections \ref{sec:picone}   and \ref{sec:grass} we present the explicit computation of the Chow quotient in the cases of Grassmannians and quadrics.

\section{Preliminaries}\label{sec:prelim}

Throughout the paper we will work over the field of complex numbers. For a free abelian group $M$, we denote by $M_{\R}$ the associated real vector space. When dealing with rational homogeneous varieties, that is, projective quotients of semisimple algebraic groups, we will freely use the marked Dynkin diagram notation, as presented, for instance, in \cite[Section~3.1]{WORS5}.

\subsection{Torus actions and quotients}\label{ssec:prelimtorus}

We recall basic facts on algebraic torus actions; see \cite{BWW,WORS2} for details. Let $H$ be an algebraic torus of rank $r$ acting nontrivially on a smooth projective variety $X$, and denote by $\Mo(H)$ its character lattice. The fixed locus decomposes as $X^H=\bigsqcup_{Y\in \cY} Y$, where each $Y$ is smooth and irreducible.

\subsubsection{Linearizations and weight polytopes}
Let $L$ be an ample line bundle on $X$. A {\em linearization of} $L$ is an $H$-action on $L$ compatible with the projection $L\to X$. Such linearizations always exist and differ by a character of $H$. For $Y\subset X^H$, the weight of the $H$-action on $L$ is constant along $Y$; we denote it by $\mu_L(Y)\in \Mo(H)$.

The {\em weight polytope of $(X,L)$} is the convex hull
\[
\Delta:=\Conv\{\mu_L(Y)\mid Y\in\cY\}\subset \Mo(H)_{\R}.
\]

The induced action on sections yields a decomposition
\[
H^0(X,L)=\bigoplus_{\tau\in \widetilde{\Gamma}} H^0(X,L)_\tau,
\]
where $\widetilde{\Gamma}$ is the set of occurring weights. The {\em polytope of sections} $\Gamma$ is the convex hull of $\widetilde{\Gamma}$. If $L$ is base point free, then $\Gamma=\Delta$ (cf. \cite[Lemma~2.4]{BWW}).

\subsubsection{GIT quotients}
For $u\in \Gamma$, define the graded algebra
\[
\cA_u := \bigoplus_{\substack{m\ge 0 \\ mu\in \Z^r}} H^0(X,mL)_{mu},
\]
which is finitely generated. The associated projective variety
\[
\GX_u := \Proj \cA_u
\]
is the {\em GIT quotient of $(X,L)$ corresponding to the linearization determined by $u$}.

\subsubsection{Chow quotients}

Let $H$ act on $X$. There exists an $H$-invariant open subset $U\subset X$ admitting a geometric quotient $\GU$, whose points correspond to closures of orbits with fixed dimension and homology class. This induces an embedding $\GU \hookrightarrow \Chow(X)$.

\begin{definition}\cite{Kapranov}\label{def:ChowQuotient}
The {\em Chow quotient} of $X$ by $H$ is the closure of this image in $\Chow(X)$. Its normalization is denoted by $\CX$. The universal family induces morphisms $p: \cU\to \CX$ and $q: \cU\to X$.
\end{definition}

A key property (cf. \cite[Theorem~0.4.3]{Kapranov}) is that $\CX$ dominates all GIT quotients:

\begin{theorem}\label{thm:invlim}
For every $v\in \Gamma$ there exists a birational morphism $\pi_v : \CX \to \GX_v$.
\end{theorem}

\subsubsection{Limit quotients}

The GIT quotients $\{\GX_v\}_{v\in \Gamma}$ form an inverse system via natural morphisms $f_{v,u}: \GX_v\to \GX_u$ whenever $X^{\ss}_v\subset X^{\ss}_u$. Although the inverse limit may be reducible, there exists a distinguished irreducible component containing $\GU$. Its normalization $\LX$ is called the {\em normalized limit quotient}.

It comes with compatible birational morphisms $\chi_v: \LX \to \GX_v$ and satisfies the following universal property:

\begin{remark}\label{rem:UniversalPropertyLimit}
Any compatible system of birational morphisms to the $\GX_v$ factors uniquely through $\LX$.
\end{remark}

In our setting, the Chow and limit quotients coincide (cf. \cite[Corollary~2.6]{BHK}):

\begin{proposition}\label{prop:BHR}
If $H$ acts on a smooth projective variety $X$, then $\CX \simeq \LX$.
\end{proposition}

\subsubsection{Combinatorial quotients}

Let $X$ be a toric variety with fan $\Sigma$, and let $H\subset T$ be a subtorus of the acting torus. The induced projection of lattices $q:\No(T)\to \No(T/H)$ defines a fan $q_*(\Sigma)$, called {\em quotient fan of $\Sigma$ by $q$}, whose cones are obtained by intersecting projected cones. The associated toric variety is the {\em combinatorial quotient} of $X$ by $H$.

\begin{proposition}\cite{KSZ}\label{prop:ChowToric}
If $X$ is a normal projective toric $T$-variety, endowed with the action of a subtorus $H\subset T$, the Chow quotient of $X$ by $H$ coincides with the combinatorial quotient.
\end{proposition}

\begin{construction}\label{cons:combiquot} 
In this paper we will be especially interested in the case of combinatorial quotients of affine spaces. In order to do so, we consider a complex torus $T$ and a $\Z$-basis of its lattice of $1$-parametric subgroups $\No(T)$. 
The $\Z$-basis determines a positive orthant $\sigma_0\subset \No(T)_\R$, which is the fan of an affine toric variety which is an affine space $U$, with a set of affine coordinates determined by the $\Z$-basis.
We then consider the associated homomorphism of character lattices $\pi:\Mo(T)_\R\to \Mo(H)_\R$; using the dual $\Z$-basis of $\Mo(T)$, and choosing a $\Z$-basis of $\Mo(H)$, we can write this map as given by a matrix, that we call {\em weight matrix}, since its columns are the weights of the $H$-action on the coordinates of $U$.  A {\em transposed Gale dual} of this weight matrix represents the associated epimorphism $q: \No(T)_\R\to \No(T/H)_\R$, and the combinatorial quotient of $U$ is defined by the quotient fan $q_*(\Sigma(\sigma_0))$ of the fan of faces of $\sigma_0$. 

We note here that the determination of the quotient fan may be computationally expensive. However, we expect that, in the case in which $U$ is an affine open neighborhood in a rational homogeneous space, one should be able to describe the quotient fan in representation theoretical fashion.
\end{construction}

\section{Affine covers}\label{sec:main}

\begin{definition}\label{def:fulldef} 
Let $H$ be an algebraic torus acting on a finite dimensional complex vector space $U$. We say that the action is {\em fully definite} if $U^H=\{0\}$ and there exists a $\Z$-basis of $\Mo(H)$ with respect to which all the weights of the action have nonnegative coordinates. An action of $H$ on a projective smooth algebraic variety is said to be {\em fully definite at} $x\in X^H$, if it is fully definite on $T_{X,x}$.  
\end{definition}

\begin{lemma}\label{lem:projective}
Let $H$ be an algebraic torus acting on a finite dimensional complex vector space $U$. Assume that the action is fully definite. Then every GIT quotient of $U$ by the action of $H$ is projective.
\end{lemma}

\begin{proof}
Let us start by considering a set of coordinates $(x_1,\dots,x_N)$ in $U$ with respect to a basis of $H$-eigenvalues of $U$. Since the action is fully definite, we may identify $H\simeq (\C^*)^n$, $\Mo(H)\simeq \Z^n$ so that we may write, for every $h\in H$:
\[h(x_1,\dots,x_N)=(h^{m_1}x_1,\dots,h^{m_N}x_N)\text{, with }m_i\in (\Z_{\geq 0})^n\setminus\{0\}.
\]

In order to compute the GIT quotients of $U$, we consider it as an affine toric variety with respect to the $N$-dimensional torus $$T=\{(x_1,\dots,x_N)|\,\, x_i\neq 0 \mbox{ for every $i$}\}\subset U.$$ Denoting by $P$ the nonnegative orthant in $\Mo(T)=\Z^N$, we have that $U=\Spec(\C[P\cap \Mo(T)])$, and, following \cite[Definition 2.16, Remark 3.6]{CML}, the GIT quotients of $U$ are in bijective correspondence with the chamber decomposition of the image of $P$ via the natural map:
$$
\pi:\Mo(T)_\R\lra \Mo(H)_\R \mbox{ (sending the canonical basis element $e_i$ to $m_i$)}.
$$
This chamber decomposition is given by the intersections of the images via $\pi$ of the faces of $P$. Given an element  $v\in \pi(P)\cap \Mo(H)$, the corresponding GIT quotient is the $T/H$-toric variety associated to the (a priori non bounded) polytope $\pi^{-1}(v)\cap P$ (see \cite[Lemma 3.2]{CML}). Thus in order to prove that they are projective, it suffices to show that $\pi^{-1}(v)\cap P$ is bounded for every $v$. 

But if $\pi^{-1}(v)\cap P$ were unbounded for some $v$, there would exist $u_0\in\pi^{-1}(v)\cap P$, and a vector $u\in \ker(\pi)$ such that $u_0+ku\in \pi^{-1}(v)\cap P$, for every $k\geq 0$, from which it would follow that $\R_{\geq 0}u$ would be contained in $P\cap \ker(\pi)$. This contradicts that, with respect to a certain set of coordinates, all the $m_i$'s have nonnegative coordinates, and no $m_i$ is equal to zero. 
\end{proof}

\begin{corollary}\label{cor:projective}
In the situation of Lemma \ref{lem:projective}, the combinatorial quotient $\CU$ by the action of $H$ is the inverse limit of the GIT quotients of $U$ by the action of $H$, and it is projective.
\end{corollary}

\begin{proof}
Note that, after some translations of notation, the fact that the combinatorial quotient $\CU$ is the limit quotient of $U$ by $H$ follows from \cite[Theorem 1.1]{CML}. 
Finally $\CU$ is projective because as a toric variety, its associated polytope is the Minkowski sum of the polytopes of its GIT quotients (cf. \cite[Definition 2.2]{CML}).  
\end{proof}

\begin{remark}\label{rem:contract}
Note that 
the morphisms from $\CU$ to geometric quotients of $U$ are toric contractions.
\end{remark}

\begin{lemma}\label{lem:BBfulldef}
Let $H$ be an algebraic torus acting on a smooth projective variety $X$, and $x\in X$ be a fixed point at which the action is fully definite. Then there exists a unique $H$-invariant open neighborhood $U_x\subset X$ of $x$ such that:
\begin{itemize}
\item there exists an $H$-equivariant isomorphism $U_x \simeq T_{X,x}$;
\item $U_x$ is the set of points $y\in X$ such that $x\in \ol{H y}$.
\end{itemize}
\end{lemma}

\begin{proof}
The existence of an $H$-invariant open neighborhood $U_x\subset X$ of $x$ which is  $H$-equivariantly isomorphic to $T_{X,x}$ is guaranteed by \cite[Theorem 2.5]{BB}. Since the action is fully definite on $x$, it is then obvious that $0\in T_{X,x}$ is contained in the closure of every $H$-orbit in $T_{X,x}$, hence it holds also that $x\in \ol{H y}$ for every $y\in U_x$. Finally, every point $y\in X$ such that  $x\in \ol{H y}$ belongs to $U_x$ because $U_x$ is an $H$-invariant neighborhood of $x$. 
\end{proof}

\begin{definition}\label{def:boundx}
In the assumptions of Lemma \ref{lem:BBfulldef}, the second property implies that $U_x$ is unique, and we call it the {\em Bruhat neighborhood of $x$}. Since $X$ is assumed to be smooth, the complement $X\setminus U_x$ is a finite union of prime Cartier divisors 
\end{definition}

\begin{lemma}\label{lem:ampleb}
Let $H$ be an algebraic torus acting on a smooth projective variety $X$, and $x\in X$ be a fixed point at which the action is fully definite. Then there exists an ample divisor $D$ on $X$ supported on the $H$-horizon of $x$. 
\end{lemma}

\begin{proof}
Consider a downgrading of the action of $H$ to a $\C^*$-action with sink $x$, and let $D'$ be a very ample effective divisor not containing $x$. Then 
 $D'_+= \lim_{t \to 0} tD'$  is a $\C^*$-invariant effective divisor in the linear system $|D'|$ not containing $x$. We notice that $D'_+$ is also $H$-invariant, hence $D'_+ \cap U_x= \emptyset$. Then, denoting by $D_1,\dots,D_k$ the irreducible components of the $H$-horizon of $x$, we may conclude by taking $D=mD'_++\sum_{i=1}^k D_i$ for $m>>0$. 
\end{proof}

\begin{proposition}\label{prop:maptocomb} 
Let $X$ be a smooth projective variety, endowed with the action of an algebraic torus $H$. Assume that the action is fully definite at a fixed point $x\in X^H$. Then every GIT quotient of $U_x$ is also a GIT quotient of $X$, and in particular there exists a contraction: 
\[\CX\lra \CU_x.\]
\end{proposition}

\begin{proof}
Let us remark that, since the Picard group of the affine space $U_x$ is trivial, then a GIT quotient of $U_x$ is determined by a linearization $v$ of the action of $H$ on $\cO_{U_x}$. 
As in the proof of Lemma \ref{lem:projective} we consider the structure of $U_x\simeq T_{X,x}$ as a toric variety with respect to an open torus $T\subset U_x$, and isomorphisms $T\simeq (\C^*)^{\dim X}$, $H\simeq (\C^*)^n$ so that the natural inclusion $H\subset T$ induces a homomorphism $\pi:\Z^{\dim(X)}\to \Z^n$ given by a matrix whose columns are nonzero and have nonnegative coefficients.
The $T$-action on the polynomial ring $\HH^0(U_x,\cO_{U_x})$ decomposes it as 
$$
\HH^0(U_x,\cO_{U_x})\simeq\bigoplus_{m\in P\cap \Z^n}\C \chi^m,
$$
where $P$ denotes the nonnegative orthant in the vector space $\R^n\simeq\Mo(T)\otimes_\Z\R$.
A linearization of the action is then determined by an element $v\in \pi( P\cap \Z^n)$, and the corresponding GIT quotient $(\GU_x)_v$ is the $T/H$ toric variety whose moment polytope is $\pi^{-1}(v)\cap P$. From the proof of Lemma \ref{lem:projective}, we know that $\pi^{-1}(v)\cap P$ is bounded. The unstable points for this linearization are precisely the points $y\in U_x$ such that $\chi^m(y)=0$ for every $m\in \pi^{-1}(v)\cap P\cap \Z^n$.

On the other hand, we consider the ample divisor $D$ supported on the $H$-horizon of $x$ provided by Lemma \ref{lem:ampleb}, and the vector spaces 
\[\HH^0(X,rD)=\{f\in\C(X)|\,\,\div(f)+rD\geq 0\}, \quad r\geq 0.\] 
We have that:
\[
\begin{array}{l}
\HH^0(X,rD)\subset \HH^0(X,(r+1)D),\mbox{ for every }r,\\[4pt]
\displaystyle\HH^0(U_x,\cO_{U_x})=\sum_{r\geq 0}\HH^0(X,rD).
\end{array}
\]
Then there exists  $r>0$ such that:
\[
\bigoplus_{m\in \pi^{-1}(v)\cap P\cap \Z^n}\C \chi^m\subset \HH^0(X,(r-1)D).
\]
The linearization of $v$ on $\cO_{U_x}$ extends to a linearization of the $H$-action on the line bundle $\cO_X(rD)$, and the corresponding space of invariant sections is still generated by $\{\chi^m|\,\, m\in \pi^{-1}(v)\cap P\cap \Z^n\}$. But, by our choice of $r$, all these sections vanish on the support of $D$, that is, the $H$-horizon of $x$ in $X$. It then follows that the set of semistable points with respect to $v$ in $U_x$ and on $(X,D)$ is the same, from which it follows that $\GU_x$ is a GIT quotient of $X$, polarized with the line bundle $D$. This concludes the proof of the first part.

Now, since we have shown that every GIT quotient of $U_x$ is also a GIT quotient of $X$, it follows that $\CX$, which is the limit quotient of $X$, admits a birational, dominant morphism to the limit quotient of $U_x$, which is $\CU_x$. Since both $\CX$ and $\CU_x$ are normal, the map is a  contraction. 
\end{proof}

Note that, by construction, the quotients $\CU_x$, $x\in X^H$, are birationally equivalent, hence there exists an irreducible subvariety $V\subset \prod_{x\in X^H}\CU_x$ dominating birationally every $\CU_x$, and the induced birational map $\CX\to \prod_{x\in X^H}\CU_x$ factors through the normalization $\ol{V}$ of $V$: 
\[
\CX \lra \ol{V}\stackrel{\mbox{\tiny norm.}}{\lra} V\subset \prod_{x\in X^H}\CU_x.
\]
Then the proof of Theorem \ref{thm:main} is concluded by showing:
\begin{proposition}\label{prop:maptocomb2}
Let $X$  be a  smooth projective variety, endowed with the action of an algebraic torus $H$, fully definite on every fixed point. Then the induced map 
$\CX\to \ol{V}$ is an isomorphism. 
\end{proposition}

\begin{proof}
Since $\CX\to \ol{V}$ is a birational map among normal varieties, it is enough to show that it is finite. If it contracted a curve $C$, then that curve would be contracted by the morphisms $\CX\to \CU_x$, for every $x\in X^H$. This would be saying that the cycles parametrized by $C$ have the same irreducible components in every $U_x$, contradicting that the $U_x$'s form a covering of $X$.
\end{proof}

The hypotheses of the above Proposition are fulfilled  when $X$ is a rational homogeneous space, that is the quotient $G/P$ of a semisimple algebraic group by a parabolic subgroup by the action of a maximal torus $H\subset P$. In fact, in this case the fixed points of the action are known to be the elements of the form $\sigma B$, where $\sigma$ is an element of the Weyl group $W=\No_G(H)/H$. The weights of the action --with respect to a linearization of an ample line bundle $L(w)$ associated to a weight $w\in \Mo(H)$-- are of the form $-\sigma(w)\in \Mo(H)$, $\sigma\in W$. We can then easily see that these values are the vertices of the polytope they generate. 

Given an element $\sigma\in W$, multiplication by (a representative in $\No_G(H)$ of) $\sigma^{-1}$ provides an isomorphism $\sigma^{-1}:U_{\sigma B}\to U_{eB}$, that descends to an isomorphism $\CU_{\sigma B}\to \CU_{eB}$, which we still denote by 
$\sigma^{-1}$. Composing this with the birational map $\CU_{eB}\dashrightarrow \CU_{\sigma B}$ induced by the intersection, we obtain a birational automorphism of $\CU_{eB}$, that we call the {\em mutation map associated to $\sigma$}; this name is meant to suggest a relation with the concept of cluster algebra, introduced by Fomin and Zelevinsky (see \cite{FZ02,Wil14,Mora26}), which we intend to explore in the future. 

We get a group homomorphism:
\[
\mu:W\to\Bir(\CU_{eB}),
\]
that contains all the necessary information to construct $\CX$ out of $\CU_{eB}$. Finally, since the variety $\CU_{eB}$ is toric, with respect to a torus $T'$ of dimension $\dim(G/B)-\dim(H)$, it is reasonable to identify $\Bir(\CU_{eB})$ with $\Bir(T')$, or with the group of birational automorphisms of some particularly simple toric compactification of $T'$ (typically a geometric quotient of $U_{eB}$). Summing up, the process of computing the Chow quotient of $X=G/P$ by the action of $H$ can be tackled in two steps:
\begin{itemize}
\item[1.] compute the combinatorial quotient $\CU_{eB}$;
\item[2.] construct $\CX$ out of $\CU_{eB}$ by using the mutation group $\mu(W)\subset\Bir(\CU_{eB})$. 
\end{itemize}

\section{Quotients of Picard number one}\label{sec:picone}

A first comment on the procedure described above is that the birational complexity of the combinatorial quotients should reflect on the intricacy of the second step. The next statement illustrates this fact on the simplest possible situation: 

\begin{proposition}\label{prop:uniquecomb}
Let $X=G/P$ be a rational homogeneous manifold, endowed with the action of a maximal torus $H\subset P\subset G$. Assume that the combinatorial quotient $\CU_x$ of the Bruhat neighborhood of an $H$-fixed point of $X$ has Picard number one. Then $\CX\simeq \CU_x$. 
\end{proposition}

\begin{proof}
The hypothesis implies that every geometric quotient of $U_x$ is isomorphic to $\CU_x$ and, in particular, the induced maps among the geometric quotients of $U_x$ are isomorphisms. Using the action of the Weyl group of $G$, which is transitive on the set of fixed points, the same holds for the geometric quotients of any other Bruhat neighborhood of $X$. 

In order to conclude, we need to prove that also the induced maps among geometric quotients of different Bruhat neighborhoods $U_{wP}, U_{w'P}$ are isomorphisms. Without loss of generality, we may assume that $w'=r_iw$, for some $i$. 
 We consider an ample line bundle $L$ on $X$, and the corresponding weight polytope $\cP$ of $L$. We note that, by AMvsFM formula (cf. \cite[Lemma~2.2]{RW}) $\mu_L(w'P)-\mu_L(wP)$ is a multiple of $w(\alpha_i)$. Taking a rational point in $\cP$ close enough to the edge of $\cP$ joining $\mu_L(w'P)$ and $\mu_L(wP)$, we obtain a linearization of the $H$-action on a large enough multiple of $L$, such that the corresponding GIT quotient is a geometric quotient of both $U_{wP}$ and $U_{w'P}$. It then follows that the induced maps among geometric quotients of $U_{wP}, U_{w'P}$ are isomorphisms.
\end{proof}

\subsection{The projectivized tangent bundle of $\P^n$}\label{ssec:projtang}

Let us consider here the case of the rational homogeneous variety $X=\DA_n(1,n)=\P(T_{\P^n})$ parametrizing flags of points and hyperplanes in $\P^n$, endowed with the action of a maximal torus in $\SL(n+1)$. The weights of this action on the tangent space at the class $x$ of the identity are:
\[
\alpha_1,\,\alpha_1+\alpha_2,\, \dots,\,\alpha_1+\dots+\alpha_n,\,\alpha_2+\dots+\alpha_n,\,\alpha_n. 
\]
In particular the weight matrix at $T_{X,x}$ is:
\[
\begin{pmatrix}
1&1&\dots&1&0&\dots&0\\
0&1&\dots&1&1&\dots&0\\
\vdots&\vdots&\ddots&\vdots&\vdots&\ddots&\vdots\\
0&0&\dots&1&1&\dots&1
\end{pmatrix}
\]
which has as transposed Gale dual:
\[
\begin{pNiceArray}{c|c|c}
I_{n-1}& (-1)_{(n-1)\times 1}& I_{n-1}
\end{pNiceArray}. 
\]
It then follows that $\CU_x$ is isomorphic to $\P^{n-1}$, so applying Proposition \ref{prop:uniquecomb} we get: 

\begin{proposition}\label{prop:projtang}
The normalized Chow quotient of $\P(T_{\P^n})$ by the action of a maximal torus in $\SL(n+1)$ is $\P^{n-1}$. 
\end{proposition}

\subsection{Quadrics}\label{ssec:quad}

We will now consider the case of smooth quadrics, that will be denoted  by $\DB_n(1)$ (resp. $\DD_n(1)$), in the case of odd dimension $2n-1$, and $\DD_n(1)$ in the case of even dimension $2n-2$. We consider homogeneous coordinates $(x_0:\dots:x_{2n})$ (resp. $(x_1:\dots:x_{2n})$) so that the quadric is defined by the equation $x_0^2+x_1x_{n+1}+\dots+x_nx_{2n}=0$ (resp. $x_1x_{n+1}+\dots+x_nx_{2n}=0$). 

We consider the maximal torus $H$ in the  group $\SO(2n+1)$ (resp. $\SO(2n)$) acting on the above coordinates with weights $0,e_1,\dots e_n,-e_1,\dots,-e_n\in \Mo(H)\simeq\Z^n$ (resp. $e_1,\dots e_n,-e_1,\dots,-e_n$). Here the $\Z$-basis $\{e_1,\dots,e_n\}\subset \Mo(H)$ corresponds to the canonical $\Z$-basis of $\Z^n$. In both cases, we have $2n$ isolated $H$-fixed points, which are the coordinate points
$P_1, \dots, P_{2n}$.  We will prove the following:

\begin{proposition}\label{prop:combquotquad}
The Chow quotient of a smooth quadric by the action of a maximal torus in the corresponding special orthogonal group is a projective space.
\end{proposition}

\begin{proof}
We will show that the combinatorial quotient of the tangent space of the quadric at a fixed point $x$ is a projective space, which allows us to conclude the proof by Proposition \ref{prop:uniquecomb}.

We will show the statement in the case of the odd dimensional quadric $\DB_{n}(1)$ and explain how to adapt the argument to the even dimensional case. 
Without loss of generality we may assume that $x=P_1=(0:1:0:\dots:0)$. Then the weights of the $H$-action on $T_{\DB_n(1),x}$ are:
\begin{equation}\label{eq:weightsodd}
-e_1,e_2-e_1,\dots,e_n-e_1,-e_2-e_1,\dots,-e_n-e_1\in\Mo(H).
\end{equation}
Then the combinatorial quotient of $T_{\DB_n(1),x}$ is the toric variety defined by the quotient fan of the positive orthant $\sigma_0$ in $\R^{2n-1}$ into $\R^{n-1}$ given by the linear surjection $q$ associated to the transposed Gale dual of the matrix whose columns are the coordinates of the above list of weights. An easy computation shows that $q$ may be given by the $(n-1)\times(2n-1)$ matrix:
\begin{equation}\label{eq:projodd}
\begin{pNiceArray}{c|c|c}
(-2)_{(n-1)\times 1}& I_{n-1}& I_{n-1}
\end{pNiceArray}.
\end{equation}
In particular, one sees that the combinatorial quotient of $T_{\DB_n(1),x}$ is $\P^{n-1}$.

In a similar way, the linear map defining the quotient fan in the case of $\DD_n(1)$ is given by the $(n-2)\times (2n-2)$ matrix:
\begin{equation}\label{eq:projeven}
\begin{pmatrix}
-1& 1&0&\ldots&0&-1& 1&0&\ldots&0\\
 0&-1&1&\ldots&0& 0&-1&1&\ldots&0\\
\vdots&\vdots&\ddots&\ddots&\vdots&\vdots&\vdots&\ddots&\ddots&\vdots\\
0&0&\ldots&-1&1&0&0&\ldots&-1&1
\end{pmatrix}.
\end{equation}
The rays of the quotient fan are then generated by the first $(n-1)$ columns of the matrix; since their sum is zero, then the combinatorial quotient is $\P^{n-2}$. 
\end{proof}

\subsection{The boundary divisor of the Chow quotient of a quadric}\label{ssec:boundaryquad}

Besides the determination of the Chow quotient, a finer analysis of the combinatorial quotients and of the mutation maps considered in our two-step procedure, allows us to describe the boundary divisors of $\CX$, parametrizing reducible $H$-invariant cycles in $X$. 

We will consider the case of a quadric of dimension $2n-1$ in $\P^{2n}$; the case of the quadrics of even dimension is analogous. 

Let us start by considering the point $P_1=(0:1:0:\dots:0)$ and the combinatorial quotient $\CU_{P_1}$. Looking at the weights of the action (\ref{eq:weightsodd}) we immediately realize that we have $2^{n-1}$ GIT chambers of maximal dimension; each one is the cone generated by a set of $n$ weights:
\[-e_1,\pm e_2-e_1,\dots,\pm e_n-e_1,\]
for a choice of $(n-1)$ signs. Each of these chambers provides a geometric quotient of $T_{\DB_n(1),P_1}$. On the other hand we may identify the sets of stable points for each of these geometric quotients. We denote by $\rho_1,\rho_2^+,\dots,\rho_n^+,\rho_2^-,\dots,\rho_n^-$ the rays of the positive orthant $\sigma_0\subset\R^{2n-1}$ corresponding to the columns of the projection matrix (\ref{eq:projodd}) and we note that the rays $\rho_i^+,\rho_i^-$  project via $\pi$ to the ray generated by $(i-1)$-th coordinate vector of $\R^{n-1}$.   
The sets of stable points are of the form 
\[T_{\DB_n(1),P_1}\setminus \big(\ol{O(\rho_2^\pm)}\cup \dots \cup \ol{O(\rho_n^\pm)}\big).\] We have one of these sets of stable points for each choice of $n-1$ signs in the formula above.

\begin{figure}[h!!]
\begin{tikzpicture}[scale=1.2, font=\footnotesize]
\draw[thick] (-1,0) -- (0,1) -- (1,0) -- (0,-1) -- (-1,0);
\draw[thin] (-1,0) -- (1,0);
\draw[thin] (0,-1) -- (0,1);

\filldraw[black] (0,0) circle (0.03) node[above right] {$-e_1$};
\filldraw[black] (1,0) circle (0.03) node[right] {$e_2-e_1$};
\filldraw[black] (-1,0) circle (0.03) node[left] {$-e_2-e_1$};
\filldraw[black] (0,-1) circle (0.03) node[below] {$-e_3-e_1$};
\filldraw[black] (0,1) circle (0.03) node[above] {$e_3-e_1$};
\end{tikzpicture}
\quad
\begin{tikzpicture}[scale=1.2, font=\footnotesize]
\draw[thick] (-1,0) -- (0,1) -- (1,0) -- (0,-1) -- (-1,0);
\draw[thin] (-1,0) -- (1,0);
\draw[thin] (0,-1) -- (0,1);

\filldraw[black] (1,0) circle (0.03) node[right] {$e_2-e_1$};
\filldraw[black] (-1,0) circle (0.03) node[left] {$-e_2-e_1$};
\filldraw[black] (0,-1) circle (0.03) node[below] {$-e_3-e_1$};
\filldraw[black] (0,1) circle (0.03) node[above] {$e_3-e_1$};
\end{tikzpicture}

\caption{The GIT chamber decompositions of $T_{\DB_3(1),P_1}$, and $T_{\DD_3(1),P_1}$.}\label{fig:GITchamber}
\end{figure}
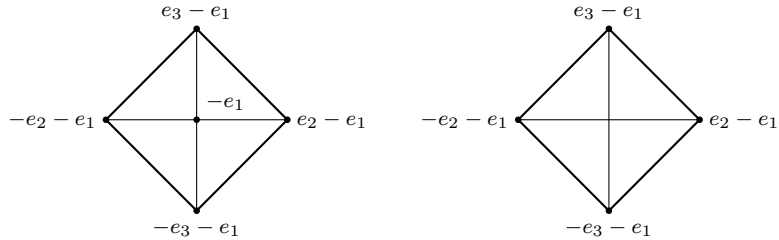

This can be interpreted as follows. We consider homogeneous coordinates $(y_1:y_2:\dots:y_n)$ in the combinatorial quotient $\CU_{P_1}\simeq\P^{n-1}$ in correspondence with the first  $n$  columns of the matrix (\ref{eq:projodd}). A general element in a hyperplane $y_i=0$, $i\geq 2$, corresponds to a reducible $H$-invariant cycle in $\C^{2n-1}$ containing two orbits of maximal dimension: one contained in $\ol{O(\rho_i^+)}$, and one contained in $\ol{O(\rho_i^-)}$.

Now we extend this description to the Chow quotient of $X$, and identify the boundary divisors in $\CX$, parametrizing reducible $H$-invariant cycles. Since each of them comes from a boundary divisor in a combinatorial quotient $\CU_{P_i}$, that we have described above, we may identify $\CX$ with $\CU_{P_1}$ and study the pullback to $\CU_{P_1}$ of the boundary divisors of any other $\CU_{P_i}$ via the natural map $\CU_{P_1}\to \CU_{P_i}$ (induced by the intersection of $U_{P_1}$ and $U_{P_i}$). 
A straightforward, somewhat tedious, computation shows that, in the case in which $i\in\{2,\dots,n\}$ the above map can be written as:
\[
(y_1:y_2:\dots:y_n)\in \P^{n-1}\mapsto \left(y_1:-\sum_{j=1}^ny_j:y_2:\dots:\widehat{y_i}:\dots:y_n\right)\in\P^{n-1}.
\]
Moreover, one can also verify that the natural map $\CU_{P_i}\to\CU_{P_{n+i}}$ is the identity for every $i$. It then follows that the boundary divisor of $\CX$, parametrizing reducible $H$-invariant cycles, consists of $n$ hyperplanes in $\P^{n-1}$:
\[(y_1+\dots+y_n)y_2\dots y_n=0.
\]

The case of $X=\DD_n(1)$ is essentially analogous. The Chow quotient $\CX$ is a projective space $\P^{n-2}$, and its boundary divisor consists of $n$ hyperplanes. Taking homogeneous coordinates $(y_2:\dots:y_n)$ in these spaces, the boundary divisor is given by:
\[(y_2+\dots+y_{n})y_2\dots y_{n}=0.
\]

\section{On Chow and combinatorial quotients of Grassmannians}\label{sec:grass}

The Chow quotient of the Grassmannians $\DA_n(k+1)$ of $k$-dimensional subspaces in $\P^n$ by the action of a maximal torus $H\subset\SL(n+1)$ has been described by Kapranov (\cite{Kapranov}) in the case $k=1$, while it is unknown for $k>1$. Let us present here some comments on how  our roadmap towards its computation would work. 

The affine covering  of $X:=\DA_n(k+1)=\bigcup_{I}U_{I}$ in Bruhat neighborhoods of fixed points  can be described in terms of Pl\"ucker coordinates with respect to a chosen set of homogeneous coordinates in $\P^n$. For every subset $I=\{i_0,\dots,i_k\}\subset\{0,\dots,n\}$ we denote by $U_I\subset \DA_n(k+1)$ the open subset where the Pl\"ucker coordinate $p_I$ is different from zero. Each $U_I$ is $H$-equivariantly isomorphic to a $(k+1)(n-k)$-dimensional affine space, and  $w(U_I)=U_{w(I)}$, for every $I\subset \{0,\dots,n\}$, $w\in W$. 

\subsection{Combinatorial quotients of Grassmannians}\label{ssec:grass1}

We will consider here the problem of computing the combinatorial quotient of $U_{I_0}$, with $I_0=\{0,\dots,k\}$. Besides some low-dimensional cases, obtaining a geometric description of the combinatorial quotient of $U_{I_0}$ is very challenging; let us here describe only some of its geometric features. Note that we have a system of affine coordinates in $U_{I_0}$:
\[
\left(u_\ell^j\right)_{\substack{j=0,\dots,k\\ \ell=k+1,\dots,n}}\in U_{I_0},
\]
completely determined by the $H$-action:
\[
h\left(u_\ell^j\right)=\left(h^{\sum_{t=j+1}^{\ell}\alpha_t}u_\ell^j\right).
\]
Alternatively, one may describe these coordinates as conveniently chosen quotients of Pl\"ucker coordinates.

As usual, we will write the weights of the $H$-action as columns of a matrix; in order to write this matrix properly, we will use the following notation:
\[
\begin{array}{l}
t_j=\alpha_1+\dots+\alpha_j\in \Mo(H)\mbox{: $j=1,\dots, n$,}\\
(0)_{m\times m'}\mbox{: $m\times m'$ zero matrix,}\\
(1)_{m\times m'}\mbox{: $m\times m'$ matrix with entries equal to $1$,}\\
\Delta^m:=\left(I_{m}\,\,|\,\,-(1)_{m\times 1}\right)\in M_{m\times(m+1)},\\
\otimes:\mbox{Kronecker product of matrices.}
\end{array}
\]
With this notation, the weight matrix of the $H$-action on $U_{I_0}$, in coordinates with respect to $\{t_1,\dots,t_n\}$, can be written as:
\begin{equation}\label{eq:weightgrass}
\begin{pNiceArray}{c|cc}
(0)_{k\times (n-k)}&-I_k\otimes (1)_{1\times (n-k)}&\\\hline
I_{n-k}&\Block{1-2}{(1)_{1\times k}\otimes I_{n-k}}&
\end{pNiceArray}.
\end{equation}
One may then check that a transposed Gale dual of this matrix is:
\begin{equation}\label{eq:Galegrass}
\Delta^k\otimes \Delta^{n-k-1}.
\end{equation}
In the case $n=5$, $k=2$, the weight matrix and its transposed Gale dual are:
\[\setlength{\arraycolsep}{4pt}
\begin{pNiceArray}{ccc|cccccc}
0 & 0 & 0 & -1 & -1 & -1 & 0 & 0 & 0 \\
0 & 0 & 0 & 0 & 0 & 0 & -1 & -1 & -1 \\
\hline
1 & 0 & 0 & 1 & 0 & 0 & 1 & 0 & 0 \\
0 & 1 & 0 & 0 & 1 & 0 & 0 & 1 & 0 \\
0 & 0 & 1 & 0 & 0 & 1 & 0 & 0 & 1
\end{pNiceArray},\quad
\begin{pNiceArray}{ccc|ccc|ccc}
1 & 0 & 1  & 0 & 0 & 0  & -1 &  0 & 1 \\
0 & 1 & -1 & 0 & 0 & 0  &  0 & -1 & 1 \\ \hline
0 & 0 & 0  & 1 & 0 & -1 & -1 &  0 & 1 \\
0 & 0 & 0  & 0 & 1 & -1 &  0 & -1 & 1 
\end{pNiceArray}
\]

Considering now the action of the subtorus $H'\subset H$ of rank $n-k$ corresponding to the submatrix  of the weight matrix consisting of the last $(n-k)$ rows, one may compute the corresponding quotient fan by looking at the transposed Gale dual of that submatrix, which is:
\[
\Delta^k\otimes I_{n-k}.
\]
For instance, in the case $n=5$, $k=2$, this reads as: 
\[
\begin{pNiceArray}{ccc|ccc|ccc}
1 & 0 & 0 & 0 & 0 & 0 & -1 & 0 & 0 \\
0 & 1 & 0 & 0 & 0 & 0 & 0 & -1 & 0 \\
0 & 0 & 1 & 0 & 0 & 0 & 0 & 0 & -1 \\\hline
0 & 0 & 0 & 1 & 0 & 0 & -1 & 0 & 0 \\
0 & 0 & 0 & 0 & 1 & 0 & 0 & -1 & 0 \\
0 & 0 & 0 & 0 & 0 & 1 & 0 & 0 & -1
\end{pNiceArray}.
\]
From this it follows that the combinatorial quotient $U_{I_0}$ by $H'$ is $(\P^k)^{n-k}$. By looking now at the first $k$ rows of the weight matrix of the $H$-action, we may conclude that the $H$-action on $U_{I_0}$ induces a $(\C^*)^k$-action on $(\P^k)^{n-k}$ that is the natural action on each of the factors. We will simply call it the {\em diagonal action of $(\C^*)^k$ on $(\P^k)^{n-k}$}. 
Summing up, we may write: 

\begin{proposition}\label{prop:niceisom}
The combinatorial quotient of a Bruhat neighborhood in the Grassmannian $\DA_n(k+1)$ is isomorphic to the normalized Chow quotient of $(\P^k)^{n-k}$ by the diagonal action of $(\C^*)^k$.
\end{proposition}

\begin{remark}\label{rem:niceisom}
Note that, in general the diagonal action of $(\C^*)^k$ on $(\P^k)^{n-k}$ does not satisfy the hypotheses of Theorem \ref{thm:main}.
\end{remark}

As a by-product of the above discussion, we may write the following statement, probably known to the experts, that we have been unable to find in the literature:
\begin{corollary}\label{cor:niceisom}
The normalized Chow quotients of $(\P^k)^{n-k}$ by the diagonal action of $(\C^*)^k$ and of $(\P^{n-k-1})^{k+1}$ by the diagonal action of $(\C^*)^{n-k-1}$ are isomorphic. In particular, the normalized Chow quotient of $\P^m\times\P^m$ by the diagonal action of $(\C^*)^m$ is the permutohedral variety.
\end{corollary}
\begin{proof}
It follows from the fact that the transposition isomorphism $\DA_n(k+1)\simeq \DA_n(n-k)$ preserves the $H$-action and the Bruhat neighborhoods.
\end{proof}

\begin{remark}\label{rem:niceisom2}
Note that the normalized Chow quotient of $(\P^k)^{n-k}$ by the diagonal $(\C^*)^k$-action has $(\P^{n-k-1})^k$ as one of its toric contractions. In fact, if $P_0$ is one of the fixed points of the natural $(\C^*)^k$-action, the $(\C^*)^k$-action on $(\P^k)^{n-k}$ is fully definite at $(P_0,\dots,P_0)$. Then, by Proposition \ref{prop:maptocomb}, we have a contraction of the normalized Chow quotient of $(\P^k)^{n-k}$ to the combinatorial quotient of the Bruhat neighborhood $(P_0,\dots,P_0)$, and one may easily check that this is $(\P^{n-k-1})^k$. The quotient map is toric because the normalized Chow quotient of $(\P^k)^{n-k}$ is a combinatorial quotient. 
Putting Corollary \ref{cor:niceisom} into the picture, the normalized Chow quotient of $(\P^k)^{n-k}$ has also a toric birational contraction to $(\P^{k})^{n-k-1}$.
\end{remark}

Whereas in the case $k=1$ the variety $\CU_{I_0}$  can be easily computed, its complete description in the case $k>1$ is a challenging problem (see \cite[p.~ 249]{GKZ}, \cite[Section 5]{BFS90}). An analysis of the case $k=2$, based on a geometric description of its boundary divisors, will be the goal of a forthcoming paper.

\subsection{Mutations}\label{ssec:grass2}

Note first that, from the transposed Gale dual (formula \ref{eq:Galegrass}) of the weight matrix of the $H$-action, we may  write an explicit coordinate expression of the rational quotient map $U_{I_0}\dashrightarrow (\C^*)^{k(n-k-1)}\subset\CU_{I_0}$:
\begin{equation}\label{eq:quotientGrass}
\setlength{\arraycolsep}{1.5pt}
\left(
u_\ell^j
\right)_{\substack{j=0,\dots,k\\ \ell=k+1,\dots,n}}\mapsto
\left(
\dfrac{u_\ell^ju_n^k}{u_\ell^ku_n^j}
\right)_{\substack{j=0,\dots,k-1\\ \ell=k+1,\dots,n-1}}
\end{equation}

On the other hand, we may also write the action of the Weyl group $W$ on $U_{I_0}$. Denoting by $r_i$ the permutation $(i-1,i)$, $i=1,\dots,n$, we may distinguish three cases: $i\leq k$, $i=k+1$, and $i>k+1$. The outcome is the following:

\[
r_i(u_\ell^j)=
\begin{cases}
\bigl(u_\ell^{r_i(j)}\bigr) & i\le k,\\[4pt]
\dfrac{1}{u_{k+1}^k}\, \left(
\begin{array}{c|ccc}
- u_{k+1}^0 & \multicolumn{3}{c}{ } \\[4pt]
\vdots & \multicolumn{3}{c}{(u_{k+1}^k\,u_\ell^j - u_\ell^k\,u_{k+1}^j)} \\[4pt]
- u_{k+1}^{k-1} & \multicolumn{3}{c}{} \\[10pt]\hline\\[-6pt]
1 & u_{k+2}^k & \cdots & u_n^k
\end{array}
\right)
 & i = k+1,\\[48pt]
\bigl(u_{r_i(\ell)}^{j}\bigr) & i > k+1.
\end{cases}
\]

Merging this with the rational quotient map described above, we obtain the mutation maps of (the torus of) $\CU_{I_0}$ that allow us to construct the Chow quotient out of the combinatorial ones.
The coordinates in the torus of $\CU_{I_0}$ will be denoted:
\[
\left(y_\ell^j\right)_{\substack{j=0,\dots,k-1\\ \ell=k+1,\dots,n-1}}.\]

By Remark \ref{rem:niceisom2}, the torus in $\CU_{I_0}$ may be compactified into a particular geometric quotient of $U_{I_0}$ isomorphic to $(\PP^{n-k-1})^k$, and so we may think of the mutation maps as birational automorphisms of $(\PP^{n-k-1})^k$. In order to do so, we use homogeneous coordinates in each factor: 
\[([z^j_{\ell}])=([z^0_{\ell}],[z^1_{\ell}],\dots,[z^{k-1}_{\ell}]),\quad\mbox{with}\quad [z^j_{\ell}]=[z^j_{k+1}:\dots:z^j_{n}],\quad\mbox{and}\quad y^j_\ell=z^j_\ell/z^j_n.
\]

Then, denoting by $r_i$ the transposition of a set of indices exchanging $i-1$ and $i$, and denoting by $p_{n-k-1}$ the projective involution of $\P^{n-k-1}$ given by:
$$
p_{n-k-1}([z_{k+1}:z_{k+2}:\dots:z_n]):=[-z_{k+1}:z_{k+2}-z_{k+1}:\dots:z_n-z_{k+1}],
$$
we may state the following: 
\begin{proposition} Let $X=\DA_n(k+1)$, $1\leq k\leq n-2$, be the Grassmannian of $k$-dimensional linear spaces in $\P^n$, with the action of a maximal torus in $\SL(n+1)$. Then the mutation group $\mu(W)$, presented as birational automorphisms of the quotient $(\P^{n-k-1})^k$ of $U_{I_{0}}$, is generated by the involutions $r_1,\dots,r_n$ given by:
\par\medskip
\begin{center}
\begin{tabular}{|C|C|}
\hline\rule{0pt}{2.5ex}
i&r_i\left([z_\ell^j]\right)\\[4pt]
\hline\hline
1\leq i\leq k-1 & \left([z_\ell^{r_i(j)}]\right)\\[4 pt]
\hline\rule{0pt}{2.5ex}
i= k & \left([z^0_{\ell}/z^{k-1}_{\ell}],\dots,[z^{k-2}_{\ell}/z^{k-1}_{\ell}],[1/z^{k-1}_{\ell}]\right)\\[4 pt]
\hline\rule{0pt}{2.5ex}
i= k+1 & \left(p_{n-k-1}([z^j_{\ell}])\right)\\[4 pt]
\hline\rule{0pt}{2.5ex}
i>k+1& \left([z^j_{r_i(\ell)}]\right)\\[4 pt]
\hline
\end{tabular}
\end{center}
\qedhere
\end{proposition}

Note that every $r_i$, $i\neq k$ is an isomorphism, while $r_k$ is not defined on the inverse image by the projection onto the last factor $\P^{n-k-1}$ of $n-k\choose 2$ codimension two linear subspaces.

\subsection{Grassmannians of lines}\label{sec:grasslines}

We consider here the particular case of the Grassmannians of lines $\DA_n(2)=\G(1,n)$, whose Chow quotients (by the action of a maximal torus $H\subset \SL(n+1)$) have been described by Kapranov. We will show here how to retrieve them by means of the tools introduced before. We will freely use here the notation previously introduced.

First of all, we may compute the combinatorial quotient of the $H$-invariant affine open subset $U_{I_0}\subset\DA_n(2)$, $I_0=\{0,1\}$. 

\begin{proposition}\label{prop:permuto}
The combinatorial quotient by the $H$-action of the affine variety $U_{I_0}$ is the projective toric variety defined by the normal fan of the $(n-2)$-dimensional permutohedron.
\end{proposition}

That projective toric variety is usually called the $(n-2)$-dimensional {\em permutohedral variety}, and can be described as the variety constructed upon $\PP^{n-2}$ by blowing up the coordinate points and then the linear subspaces generated by them in order of increasing dimension (see \cite[Section 2]{Lin24}). 

\begin{proof}
This statement is classically known, and can be proved in several ways. First of all, it could be seen as a dual statement of \cite[Example 7.3.C]{GKZ}. Moreover, by looking at the left and right hand parts of the above matrix, one could argue that the variety defined by the quotient fan is the minimal toric resolution of the standard Cremona transformation of $\P^{n-2}$, which is the permutohedral variety. 

From the discussion presented above, we may obtain it as follows. By Proposition \ref{prop:niceisom}, the combinatorial quotient of the $H$-invariant affine open set equals the Chow quotient of $(\P^1)^{n-1}$ by the diagonal action of $\C^*$, which is the $(n-2)$-dimensional permutohedral variety (see \cite[Example~7.3]{WORS6}). 
\end{proof}

Let us now discuss how to compute the Chow quotient of the Grassmannian out of the combinatorial quotient of $U_{01}$, by putting the action of the Weyl group into the picture. Following the previous section, the mutations $r_1,\dots,r_n$, presented as birational automorphisms of $\P^{n-2}$, in homogeneous coordinates $[z_2:\dots:z_n]$, read in our case as follows:

\par\medskip
\begin{center}
\begin{tabular}{|C|C|}
\hline\rule{0pt}{2.5ex}i&r_i\left([z_2:\dots:z_n]\right)\\[4pt]
\hline\hline\rule{0pt}{2.5ex}
1 & [z_2^{-1}:\dots:z_n^{-1}]\\[4pt]
\hline\rule{0pt}{2.5ex}
2 & [-z_2:\,-z_2+z_3:\,\cdots:\,-z_2+z_n]\\[4 pt]
\hline\rule{0pt}{2.5ex}
 i \ge 3&  [z_2:\cdots:z_{i}:z_{i-1}:\cdots:z_n] \\[4 pt]
\hline
\end{tabular}
\end{center}\par\medskip

Let us now consider in the geometric quotient $\PP^{n-2}$ the points whose homogeneous coordinates $[z_2:\dots:z_n]$ are the following:
\[p_1:=[1:1:\dots:1],\,\,p_2:=[1:0:\dots:0]
,\,\,\dots,\,\,p_n:=[0:0:\dots:1].\] 

\begin{proposition}\label{prop:chowquotG(1,n)}
The normalized Chow quotient of the Grassmannian of lines $\DA_n(2)$ by the action of the maximal torus $H$ is isomorphic to the variety obtained from $\PP^{n-2}$ by blowing up $\{p_1,p_2, \dots, p_n\}$ and then the linear subspaces generated by them in order of increasing dimension.
\end{proposition}

\begin{proof}
Let $X'$ be the variety obtained from $\PP^{n-2}$ blowing up $\{p_1,p_2, \dots, p_n\}$ and then the linear subspaces generated by them in order of increasing dimension. 

The Bruhat neighborhoods of the Grassmannian $X=\DA_n(2)$ that we are considering are of the form $U_{w(I_0)}$, $w\in W$, and $\CU_{I_0}$ is the permutohedral variety obtained upon $\P^{n-2}$ by blowing up $p_2,\dots,p_n$ and then the linear subspaces generated by them in order of increasing dimension. Hence it follows that the combinatorial quotient $\CU_{w(I_0)}$ is the permutohedral variety obtained upon $\P^{n-2}$ by blowing up $w(p_2),\dots,w(p_n)$ and then the linear subspaces generated by them in order of increasing dimension. 

Since the subgroup of $W$ generated by $r_2,\dots,r_n$ acts on the set of points $\{p_1,p_2,\dots,p_n\}$ as its group of permutations ($r_i$ exchanges $p_{i}$ and $p_{i-1}$), it follows that there are contractions $X'\to \CU_{w(I_0)}$ for every $w$ in $W$, that commute with the birational transformation among the $\CU_{w(I_0)}$'s. By Proposition \ref{prop:maptocomb2}, it then follows that we have a morphism $X'\to \CX$. One may easily check that there are no curves contracted simultaneously by every contraction  $X'\to \CU_{w(I_0)}$, hence it follows that $X'$ is isomorphic to $\CX$.
\end{proof}

\bibliographystyle{plain}
\bibliography{bibliomin}

\end{document}